\documentclass[12 pt]{amsart}

\usepackage{amsthm}
\usepackage{mathrsfs}
\usepackage{amssymb}
\usepackage{latexsym}
\usepackage{url}
\usepackage{amsmath}
\usepackage{verbatim}
\usepackage{graphicx}
\usepackage{epsf}
\usepackage{enumerate}
\usepackage{color}
\usepackage{geometry}
\usepackage{hyperref}
\definecolor{dark-blue}{rgb}{0,0,0.6}
\definecolor{Purple}{rgb}{0.2,0,0.25}
\hypersetup{breaklinks=true,
pagecolor=white,
colorlinks=true,
citecolor=dark-blue,
filecolor=black,%
linkcolor=Purple,%
urlcolor=black
}

\newtheorem{thm}{Theorem}[section]

\newtheorem{lem}[thm]{Lemma}
\newtheorem{defin}[thm]{Definition}

\theoremstyle{definition}

\newtheorem{remark}[thm]{Remark}

\newcommand{\R}{\mathbb{R}}
\newcommand{\Q}{\mathbb{Q}}
\newcommand{\C}{\mathbb{C}}
\newcommand{\N}{\mathbb{N}}
\newcommand{\Z}{\mathbb{Z}}
\newcommand{\T}{\mathbb{T}}
\newcommand{\re}{\textnormal{Re}}
\newcommand{\im}{\textnormal{Im}}

\newcommand{\card}{\textnormal{card}}
\newcommand{\bref}[1]{\textbf{\ref{#1}}} 
\newcommand{\beqref}[1]{\textbf{(\ref{#1})}} 

\title[Remarks on the Cauchy functional equation]{Remarks on the Cauchy functional equation and variations of it}
\author{Daniel Reem}
\address{Department of Mathematics, The Technion - Israel Institute of Technology, Haifa 3200003, Israel. Part of this work was done while the author was at the National Institute of Pure and Applied Mathematics (IMPA), Rio de Janeiro, Brazil, and at the Institute of Mathematical and Computer Sciences (ICMC), University of S\~ao Paulo, S\~ao Carlos, Brazil. } 
\email{dream@tx.technion.ac.il}

\date{February 8, 2017}

\subjclass[2010]{39B52, 43A22, 39B22, 39B05, 26B99, 22B99}
\keywords{ Complex exponent, initial value, locally measurable, periodic, regularity condition, restricted domain, stability, the Cauchy functional equation}
\begin{document}
\maketitle
\begin{abstract}
This paper examines various aspects related to the Cauchy functional equation $f(x+y)=f(x)+f(y)$, 
a fundamental equation in the theory of functional equations. In particular, it considers its  solvability and its stability relative to subsets of multi-dimensional Euclidean spaces and tori. 
Several new types  of regularity conditions are introduced, such as a one in which a complex  exponent of the  unknown function is locally measurable. An initial value approach to analyzing  this  equation is considered too and it yields a few by-products, such as the existence of a  non-constant real function having an uncountable set of periods which are linearly independent over the rationals. The analysis is extended to related equations such as the Jensen equation, the multiplicative Cauchy equation, and the Pexider equation. 
The paper also includes a rather comprehensive survey of the history of 
the Cauchy equation. 
\end{abstract}

\section{Introduction}\label{sec:Introduction}
\subsection{History} 
A well-known and fundamental equation in the theory of functional equations is 
the Cauchy functional equation 
\begin{equation}\label{eq:Cauchy}
f(x+y)=f(x)+f(y).
\end{equation}
This equation (as well as its three sisters $f(x+y)=f(x)f(y)$, $f(xy)=f(x)+f(y)$, $f(xy)=f(x)f(y)$), was introduced by Cauchy in his book \cite[p. 103]{Cauchy1821} from 1821.  Cauchy rigorously analyzed \beqref{eq:Cauchy}  under the assumptions that the unknown function $f$ is a continuous function from $\R$ to $\R$ and the variables $x$ and $y$ can be arbitrary real numbers. (Gauss \cite[p. 212]{Gauss1809} also considered a certain version of \beqref{eq:Cauchy} in his book from 1809, but the consideration was neither rigorous nor very explicit; see also \cite[pp. 47, 106--109]{Aczel}. Going back even further, to 1794, one can find in the book of  Legendre \cite{Legendre1794book}, in the section devoted to the consideration of ratios of areas of rectangles, both a use and an analysis of \beqref{eq:Cauchy}, but they are also neither rigorous nor explicit; see also \cite[pp. 369--370]{AczelDhombres}) Cauchy's equation \beqref{eq:Cauchy} has many applications 
inside the theory of functional equations and in other mathematical and scientific fields including geometry, real and complex analysis, probability, functional analysis, dynamical systems, partial differential equations, classical and statistical mechanics,  and economics. 
It is therefore nothing but natural that it has attracted the attention of 
many authors for a long period of time. Formulating this fact in more romantic terms \cite[p. 2]{Kannappan2009}, Kannappan wrote: ``Researchers fell in love with these equations [Cauchy's equation \beqref{eq:Cauchy} and its sisters mentioned above], and the romance will continue and will result in many more interesting results''.

A common path of investigation of \beqref{eq:Cauchy} is to impose 
 various types of ``regularity'' conditions on the unknown function. It turns out that  in the specific case where $f:\R\to\R$, each of these conditions implies the existence of some $c\in\R$ such that $f(x)=cx$ for all $x\in \R$, and this fact has been proved in various ways. For instance, Cauchy \cite[pp. 103--106]{Cauchy1821} assumed that $f$ is continuous,  Darboux  showed that $f$ can be assumed to be either monotone \cite{Darboux} 
 or bounded on an interval \cite{Darboux1880},  
 Fr\'echet \cite{Frechet}, Blumberg \cite{Blumberg1919},  Banach \cite{BanachCauchy}, 
 Sierpi\'nski \cite{Sierpinski1,Sierpinski2},  Kac \cite{Kac},  Alexiewicz-Orlicz \cite{AlexiewiczOrlicz}, and Figiel \cite{Figiel1969} 
 assumed that $f$ is Lebesgue measurable, 
Kormes  \cite{Kormes} assumed that $f$ is bounded  on a measurable  set of positive measure, Ostrowski \cite{Ostrowski} 
and  Kestelman \cite{Kestelman} assumed that $f$ is bounded from one side on a measurable  set of positive measure,  and Mehdi \cite{Mehdi} assumed that $f$ is bounded above on a second category Baire set. On the other hand, in 1905 Hamel \cite{Hamel} investigated \beqref{eq:Cauchy} without any further condition of $f$. By using the Hamel basis, which was introduced for this  purpose, he showed that there are nonlinear solutions to \beqref{eq:Cauchy} and he found all of them (see also Remark \bref{rem:Nonlinear} below). 

The functional equation \beqref{eq:Cauchy} has been generalized or modified in many other directions. A typical path 
in this spirit is to take the domain and range of  $f$ to be groups of a certain kind, e.g., (locally compact) Polish groups, and to prove that if $f$ satisfies a certain kind of measurability assumption  (say Baire, Haar, or Christensen), and maybe additional assumptions, then it must be continuous (and, depending on the setting, possibly linear \cite[pp. 36--38]{AczelDhombres}). See, among others, Banach \cite[p. 23]{BanachBook}, Pettis \cite{PettisCauchy}, Hewitt-Ross \cite[p. 346]{HewittRoss}, Monroe \cite{Monroe}, Figiel \cite{Figiel1969}, Itzkowitz  \cite{Itzkowitz}, Baker \cite{Baker}, Brzdek \cite{Brzdek},  Christensen \cite{Christensen}, Gajda \cite{Gajda1986}, Grosse-Erdmann \cite{Grosse-Erdmann1989jour}, J{\'a}rai-Sz{\'e}kelyhidi \cite{JaraiSzekelyhidi1996jour}, Neeb \cite{Neeb}, and Rosendal \cite{Rosendal}. In some of these generalizations the form of the functional  equation is more general than \beqref{eq:Cauchy}. (Remark: the terminology used in some of the references mentioned above and below is sometimes different from that used here, and, in particular, instead of speaking on the Cauchy functional equation, from time to time one speaks on homomorphism or additive functions; below we will sometimes also call solutions of \beqref{eq:Cauchy} ``additive functions''.)

A different path  is to impose assumptions  different from measurability, e.g.,  algebraic ones. See the book of Dales  \cite{Dales2000book} for an extensive discussion on this issue. Yet another path of generalization/modification is to change the domain of  definition of $f$ 
so that it will no longer have a necessarily nice algebraic structure but rather 
it will merely be a certain subset of the original domain, e.g., an orthant, a convex set,  a 
complement of a measure zero set, and so forth. A variation of this is to 
change the domain of validity of the equation, for instance, to assume that $f$ satisfies \beqref{eq:Cauchy} only for pairs $(x,y)$ belonging to a subset of $\R^{2n}$, e.g., a manifold 
(and $f$ may be defined on  the whole space or on a subset of it). 
In both cases one may conclude the existence of nonlinear solutions to \beqref{eq:Cauchy} 
 (even under ``strong'' regularity conditions) or that (in the case where $f$ is assumed to be 
 defined on the whole space) $f$ must satisfy \beqref{eq:Cauchy} for all possible pairs $(x,y)$. 
  For a very partial list of related works see Pisot-Schoenberg \cite{PisotSchoenberg1964}, 
Aczel-Erd\"os \cite{AczelErdos1965}, Jurkat \cite{Jurkat1965}, de Bruijn \cite{deBruijn1966}, 
Zdun \cite{Zdun1972jour}, Kuczma \cite{Kuczma1973}, Dhombres-Ger \cite{DhombresGer1978jour}, 
Dhombres \cite{Dhombres1979}, Forti \cite{Forti1983}, Matkowski \cite{Matkowski1985},  Sablik \cite{Sablik1990}, Szabo \cite{Szabo1993}, 
Choczewski et al. \cite{CCR1994}, Alsina and Garc Roig \cite{AGR1994}, 
 Ger-Sikorska \cite{GerSikorska1997}, 
Skof \cite{Skof2004}, Paneah \cite{Paneah2004,Paneah2007}, Shalit \cite{Shalit2005,Shalit2009}, 
Laohakosol-Pimsert \cite{LaohakosolPimsert2011} 
and the references therein.

More information related to \beqref{eq:Cauchy}, including many additional references, can  be found, for instance, in the books of Acz\'el \cite{Aczel}, Acz\'el-Dhombres \cite{AczelDhombres} (both books also mention 
 applications of \beqref{eq:Cauchy} in various scientific areas), Czerwik \cite{Czerwik2002book} (the book contains also a treatment of the set valued version of Cauchy's equation), Einchorn \cite{Einchorn1978} (for applications 
 in economics), J{\'a}rai  \cite{Jarai} (various regularity conditions), Kannappan \cite{Kannappan2009}, 
 Kuczma \cite{Kuczma2009book},  and the surveys of Kannappan \cite{Kannappan1989} and Wilansky \cite{Wilansky}. See also Section \bref{sec:Stability} below regarding a short survey and several  references related to stability in the context of the Cauchy equation. 

 \subsection{Contribution and paper layout} 
 This paper examines various aspects related to \beqref{eq:Cauchy}, including its solvability and its stability under new kind of assumptions, and presents new methods to analyze it. The first result that is 
 proved (Section \bref{sec:Measurable} below) is the following: 
 given $n\in\N$, if $f:\R^n\to \R$ satisfies \beqref{eq:Cauchy} for every $x,y\in\R^n$ and if a complex exponent of $f$ is locally  measurable, 
  that is,  $e^{if}$ is Lebesgue measurable on a certain open ball (and hence on a certain non-degenerate and compact hypercube contained in that ball), then $f$ must be linear,  
 namely, there exists $c\in\R^n$ such that $f(x)=c\cdot x$ for all $x\in\R^n$   
 (here, of course,  $c\cdot x$ denotes the inner product between the vectors $c=(c_k)_{k=1}^n$  and $x=(x_k)_{k=1}^n$, that is, $c\cdot x:=\sum_{k=1}^n c_k x_k$). This regularity condition is strictly weaker than measurability (Remark \bref{rem:e^if} below) and although it seems to be new, traces of it can be detected in the paper of Kac \cite{Kac} (Remark \bref{rem:Kac} below). In Section \bref{sec:IVP} 
 the same result is proved using a different approach by considering  \beqref{eq:Cauchy} as an initial value problem and utilizing periodicity properties which arise naturally. 
 Two by-products of this approach are the observation that there exist nonlinear solutions of \beqref{eq:Cauchy} having an uncountable dense set of periods which are (the periods) linearly  independent over the rationals (Remark \bref{rem:Nonlinear} below) and the observation that any  solution to \beqref{eq:Cauchy} 
 defined on a multidimensional topological (flat) torus  must vanish identically (Section \bref{sec:Torus} below).  
The previous analysis is extended to other settings: 
 to Jensen's equation  (Section \bref{sec:Jensen} below), 
 to Cauchy's equation on restricted domains 
 satisfying a certain abstract algebraic condition and related conditions (Section \bref{sec:Restricted} below), 
 to multiplicative Cauchy's equation (Section \bref{sec:Multiplicative} below), 
  to an alternative version of Cauchy's equation 
 (Section \bref{sec:Alternative} below), to Pexider's equation (Section \bref{sec:Pexider} below), 
and to the stability of Cauchy's equation (Section \bref{sec:Stability} below). 
We conclude the paper in Section \bref{sec:Remarks} with a few remarks regarding possible additional extensions, 
such as to a setting in which the  regularity condition is abstract (Remark \bref{rem:Abstract} below). The paper also has an appendix (Section \bref{sec:Appendix} below) which contains proofs of several assertions.

\section{The case where $e^{if}$ is locally measurable, $\R^n$}\label{sec:Measurable}
 This section contains the main theorem of the paper (Theorem \bref{thm:CauchyMeasurable} below), as well as several remarks related to it (Remarks \bref{rem:e^if}--\bref{rem:Kac} below). 
\begin{thm}\label{thm:CauchyMeasurable}
If $f:\R^n\to \R$ satisfies \beqref{eq:Cauchy} for every $x,y\in\R^n$ and if $e^{if}$ is locally Lebesgue measurable, then there exists some $c\in \R^n$ such that $f(x)=c\cdot x$ for all $x\in \R^n$.
\end{thm}

\begin{proof}
By applying the function $t\mapsto e^{it}$, $t\in\R$ on both sides of \beqref{eq:Cauchy} we have  $e^{if(x+y)}=e^{if(x)}e^{if(y)}$ 
 for all $x,y\in \R^n$. Let $I$ be any hypercube on which $e^{if}$ is known to be measurable. By taking an arbitrary $y\in\R^n$ and $x\in I$ and using the fact that the function $x\mapsto e^{if(x)}e^{if(y)}$, $x\in I$ is measurable on $I$, we see that $e^{if}$ is  measurable on any translated  copy of $I$. Because $\R^n$ can be represented as a countable union of translated copies of $I$, we conclude that $e^{if}$ is measurable on $\R^n$. 

Now, from a well-known theorem about the continuity of measurable homomorphisms acting between certain groups 
(in particular, between a locally compact  group and the circle group: see  \cite[p. 346]{HewittRoss} for an even more general statement), it follows that the measurable homomorphism $e^{if}$ is continuous. Since its image is 
obviously contained in the unit circle of the complex plane $\C$, it follows that $e^{if}$ is a character (namely, a continuous homomorphism from $\R^n$ to the unit circle in $\C$).
 By a well known representation result regarding the character group of $\R^n$ (see, e.g., \cite[pp. 366--368]{HewittRoss}) it follows that there exists $c\in \R^n$ such 
that 
\begin{equation}\label{eq:Character}
e^{if(x)}=e^{ic\cdot x},\quad \forall x\in \R^n. 
\end{equation}
Alternatively,  one can deduce \beqref{eq:Character} 
 without referring to the theory presented in \cite{HewittRoss} by using 
 an approach based on Fourier series \cite{ReemCharacterFourier}. 
 
Anyway, because of \beqref{eq:Character} there exists a function $k$ from $\R^n$ to the set of  integers such that $f(x)=c\cdot x+2\pi k(x)$ for all $x\in \R^n$. Since $f$ satisfies \beqref{eq:Cauchy} for every $x,y\in\R^n$ 
and since \beqref{eq:Cauchy} is linear, we conclude that $k$ also satisfies \beqref{eq:Cauchy} for all $x,y\in\R^n$. It remains to show that $k$ vanishes identically. Assume that this is not true, 
namely $k(x)\neq 0$ for some fixed $x$. Then by \beqref{eq:Cauchy} (with $k$ instead of $f$) we have $k(x/m)=k(x)/m$ for any positive integer $m$. In particular this is true for $m>2|k(x)|$, a  contradiction since $0<|k(x)/m|<0.5<1$  but $k(x/m)$ is an integer. Thus $k\equiv 0$ and $f(x)=c\cdot x$ for each $x\in\R^n$. Finally, an immediate verification shows that $f$, having this linear form, still satisfies \beqref{eq:Cauchy} for every $x,y\in\R^n$. 
\end{proof}

\begin{remark}\label{rem:e^if}
 The regularity condition imposed on $f$ is strictly weaker than merely measurability. 
 Indeed, if $f$ is Lebesgue measurable, then so is $e^{if}$ as a composition of $f$ with the  complex exponent function 
 $t\mapsto e^{it}$, $t\in\R$ which is continuous. On the other hand, even if $e^{if}$ is a continuous (multiplicative) homomorphism,  then $f$ itself need not be  measurable as the following simple  example shows:  for each integer $m$ let $A_m$  be a  non-measurable set of $[m,m+1)$,  and let $f:\R\to\R$ be defined by $f(x):=2\pi m$ if $x\in A_m$, and $f(x):=-2\pi (m+1)$ if $x\in [m,m+1)\backslash A_m$. 
 
 Another example, perhaps more interesting (see also Remark \bref{rem:exp_fm} below), is $f(x):=2\pi [g(x)],\,x\in\R$ where 
 $[t]$ is the greatest integer not exceeding the real number $t$ and $g$ is an arbitrary nonlinear solution of \beqref{eq:Cauchy}. 
 Indeed, $e^{if}\equiv 1$ but there can be no set of positive measure on which $f$ is bounded, since if 
 $A$ is such a set, then also $g$ is bounded on $A$ and by the theorem of Kormes \cite[Theorem I]{Kormes} 
 it follows that $g$ is linear, a contradiction. Since any measurable function $h:\R\to\R$ is bounded 
 on some set of positive measure (simply on $A_m:=\{x: h(x)\in [m,m+1)\}$ for some integer $m$ because 
 a measurable function is finite almost everywhere and hence $\bigcup_{m\in\Z}A_m$ 
 is, up to a set of measure 0, the whole space) it follows that $f$ is not measurable. 
\end{remark} 

\begin{remark}\label{rem:Kac}
The regularity condition that $e^{if}$ is measurable has traces in the 1937 paper of Kac \cite{Kac}. Indeed, while he assumed that $f:\R\to\R$ was measurable, in the proof he implicitly  used only the  measurability of  $e^{if}$. Modifications of Kac's proof can be found in other places, e.g., in \cite[pp. 54--56]{AczelDhombres}, \cite[pp. 39--40]{Czerwik2002book} (in these places, however, $f$ is  from $\R$ to $\C$ and it satisfies the multiplicative Cauchy equation  \beqref{eq:CauchyMult} below rather than the additive one \beqref{eq:Cauchy} above). Anyway, in all of these places (including in Kac's paper) the proofs seem to be restricted to the one-dimensional case and it is explicitly assumed that $f$ is measurable and not that $e^{if}$ is measurable. 
\end{remark}

\section{The initial value approach}\label{sec:IVP}
In this section a different proof of Theorem \bref{thm:CauchyMeasurable} is presented. 
The approach is to consider \beqref{eq:Cauchy} as an initial value problem. The inspiration comes from differential equations. The derivation is based on two lemmas. 
The first (Lemma \bref{lem:period2} below) is known. For the sake of completeness  and because it is not easy to find in the literature a proof of this lemma for the case $n>1$, a proof is provided in the appendix (Section \bref{sec:Appendix} below). The second lemma (Lemma \bref{lem:UniqueRn} below) contains a key technical assertion. 
 
In Lemma \bref{lem:period2} below (and elsewhere) we use the following terminology and observations. A (closed) hyper-parallelepiped spanned by some basis  is the set $I:=\{u_0+\sum_{k=1}^n t_k u_k:\,\,(t_k)_{k=1}^n\in [0,1]^n\}$  where $U:=\{u_1,\ldots,u_n\}$ is a given  basis of $\R^n$ and $u_0\in\R^n$ is a given point. A right-semi-open hyper-parallelepiped is the  set $I:=\{u_0+\sum_{k=1}^n t_k u_k: \,(t_k)_{k=1}^n\in [0,1)^n\}$, namely a hyper-parallelepiped from which the right facets have been removed. Of course, a (closed or right-semi-open) hypercube is a (closed or right-semi-open, respectively) hyper-parallelepiped spanned by an orthogonal basis, where the lengths of all the basis elements are equal. We say that a function $h:\R^n\to\C$ is $U$-periodic (or periodic with respect to the basis $U$) if it satisfies  $h(x+u_k)=h(x)$ for each $x\in \R^n$ and each $k\in \{1,\ldots,n\}$.

\begin{lem}\label{lem:period2}
Let $I\subset \R^n$ be a given (closed or semi-open) hyper-parallelepiped spanned by some basis $U:=\{u_1,\ldots,u_n\}$ of $\R^n$. Suppose that $h:\R^n \to \C$ is $U$-periodic. If $h$ is Lebesgue integrable on $I$, then for each $y\in \R^n$ the translation $h_y:\R^n\to\R$ defined by $h_y(x):=h(x+y)$ for all $x\in\R^n$ is Lebesgue integrable on $I$ and $\int_I h_y(x)dx=\int_I h(x)dx$. 
\end{lem}

\begin{lem}\label{lem:UniqueRn}
Let $g:\R^n\to \R$. Suppose that $I\subset \R^n$ is a hypercube spanned by some basis $U:=\{u_1,\ldots,u_n\}$. Assume that $e^{ig}$ is Lebesgue measurable on $I$ and that $g$ is $U$-periodic. If $g$ satisfies \beqref{eq:Cauchy} for every $x,y\in\R^n$, then $g$ vanishes identically.
\end{lem}
\begin{proof}
Since $g$ satisfies \beqref{eq:Cauchy} for every $x,y\in\R^n$ and since $e^{ig}$ is measurable on some hypercube $I$, we can use the same reasoning given in   Section \bref{sec:Measurable} above (the beginning of the first proof of Theorem \bref{thm:CauchyMeasurable}) to conclude that $e^{ig}$ is measurable on $\R^n$ and hence on any hypercube contained in $\R^n$. Given any rational number $\alpha$, let $h_{\alpha}:\R^n\to\C$ be defined by $h_{\alpha}(x):=e^{ig(\alpha x)}$ for every $x\in\R^n$. The change of variables formula for scaling \cite[pp. 170--171]{Jones2001book} and the previous discussion on the measurability of $e^{ig}$ imply that $h_{\alpha}$ is  measurable on $\R^n$, and, being bounded, it is Lebesgue integrable over any hypercube. In particular $h_{\alpha}$ is measurable over $I_0:=\{\sum_{k=1}^n t_k u_k:\,\,(t_k)_{k=1}^n\in [0,1)^n\}$. In addition, we have 
\begin{equation}\label{eq:h_alpha}
h_{\alpha}(x)=e^{i\alpha g(x)},\quad \forall x\in\R^n,\,\forall\,\alpha\,\,\textnormal{rational}
\end{equation}
because $g(\alpha x)=\alpha g(x)$ for each rational $\alpha$ and each $x\in\R^n$ (this identity is, of course, a well known fact which easily follows from the assumption that $g$ satisfies \beqref{eq:Cauchy} for every $x,y\in\R^n$: see \cite[Theorem 5.2.1, pp. 128--129]{Kuczma2009book}).

Let $\alpha>0$ be some rational satisfying $\int_{I_0} h_{\alpha}(x)dx\neq 0$. It will be  shown in a moment that such an $\alpha$ does exist.  Since $g$ is $U$-periodic, the identity \beqref{eq:h_alpha} implies that $h_{\alpha}$ is also $U$-periodic. This fact, the fact that $I_0$ is also spanned by $U$, and Lemma \bref{lem:period2} (with $h_{\alpha}$ and $I_0$ instead of $h$ and $I$, respectively), imply that $\int_{I_0} h_{\alpha}(x)dx=\int_{I_0}h_{\alpha}(x+y)dx$ for every $y\in\R^n$. But \beqref{eq:h_alpha} and \beqref{eq:Cauchy} imply that 
$\int_{I_0}h_{\alpha}(x+y)dx=\int_{I_0}e^{i\alpha g(x+y)}dx=\int_{I_0}e^{i\alpha g(x)}e^{i\alpha g(y)}dx=e^{i\alpha g(y)}\int_{I_0}h_{\alpha}(x)dx$ for all $y\in\R^n$. We conclude that $\int_{I_0}h_{\alpha}(x)dx=e^{i\alpha g(y)}\int_{I_0}h_{\alpha}(x)dx$ for all $y\in\R^n$. As a result of this equality and the assumption that $\int_{I_0} h_{\alpha}(x)dx\neq 0$ it follows that $1=e^{i\alpha g(y)}$ for all $y\in \R^n$. Therefore $\alpha g(y)\in \{2\pi k: k\in \Z\}$ for any $y\in \R^n$. Assume for a contradiction that $\alpha g(y_0)\neq 0$ for some $y_0$. We can write $2\pi k_0=\alpha g(y_0)$ for some  $0\neq k_0\in \Z$. Since $k_0\neq 0$ and because $g(qy_0)=qg(y_0)$ for each rational $q$, we have $\alpha g(y_0/(7k_0))=\alpha(1/(7k_0))g(y_0)=2\pi k_0/(7k_0)=2\pi/7$. Thus $\alpha g(y)\notin \{2\pi k: k\in \Z\}$ for $y:=y_0/(7k_0)$, a contradiction. Hence $g\equiv 0$. 
 
 It remains to prove that $\int_{I_0} e^{i\alpha g(x)}dx\neq 0$ for some positive rational $\alpha$.  If this is not true, 
 then in particular $\int_{I_0} e^{i \alpha g(x)}dx=0$ for all rationals $\alpha\in (0,1]$.  As a result  
\begin{equation*}
0=\int_{I_0}e^{i\alpha g(x)}dx=\int_{I_0}e^{i g(\alpha x)}dx=\int_{\alpha I_0} e^{i g(w)}dw/\alpha^n=\int_{\alpha I_0} e^{i g(w)}dw
\end{equation*} 
for each  rational $\alpha\in (0,1]$ by  the change of variables formula for scaling. 
 Therefore, given $y\in\R^n$ and a rational number $\alpha\in (0,1]$, we have 
\begin{equation*}
\int_{y+\alpha I_0}e^{i g(x)}dx=\int_{\alpha I_0}e^{i g(z+y)}dz=e^{ig(y)}\int_{\alpha I_0}e^{i g(z)}dz=0. 
\end{equation*}
It follows in particular that $\int_{S}e^{i g(x)}dx=0$ whenever $S$ is a dyadic copy of $I_0$, namely, $S$ is obtained 
from $I_0$ by scaling it by $2^{-m}$ for some nonnegative integer $m$ and translating it by a vector of 
the form $y=\sum_{k=1}^n (m_k/2^{m})u_k$, $m_k\in \Z$. The above is true even if we remove from, or add to,  
these dyadic copies of $I_0$ several parts of their (measure 0) boundaries. It is well-known that any nonempty open set $S$ of $\R^n$  is a countable union of pairwise disjoint dyadic copies of $I_0$, namely $S=\cup_{m=1}^{\infty}S_m$ where $S_m$ is a dyadic copy of $I_0$ for each $m\in\N$, and $S_m\cap S_p=\emptyset$ whenever $m\neq p$ (see \cite[p. 50]{Rudin1987book}; the proof given there is for $I_0=[0,1)^n$, namely $U$ is the canonical basis, but the same proof holds for an arbitrary hypercube, or, more generally, an arbitrary hyper-parallelepiped, simply by representing the vectors in $\R^n$ using the given basis $U$). This fact, combined with the  countable additivity of the integral, implies that for each nonempty open set $S$ we have 
$\int_S e^{ig(w)}dw=\sum_{m=1}^{\infty}\int_{S_m}e^{ig(w)}=\sum_{m=1}^{\infty}0=0$. 

Now, given a measurable subset $S\neq\emptyset$ of $I_0$, a simple consequence of the (upper) regularity of the Lebesgue  measure $\mu$ (see \cite[Property M9, p. 50]{Jones2001book}) implies  that there exists a decreasing sequence  $(U_m)_{m=1}^{\infty}$ of open subsets of $\R^n$   satisfying $S\subseteq U_m$ and  $\mu(U_m\backslash S)<1/m$ for each $m\in\N$. For each $A\subseteq \R^n$ let $1_A$ be the characteristic function of $A$, namely $1_A(x):=1$ if $x\in A$ and $1_A(x):=0$ otherwise. It must be that 
 $\lim_{m\to \infty}1_{U_m}(x)=1_S(x)$ for each $x\notin (\cap_{m=1}^{\infty}U_m)\backslash S$. Indeed, if $x\in S$, then both sides are equal to 1 since $1_{U_m}(x)=1_S(x)=1$, and if $x\notin \cap_{m=1}^{\infty}U_m$, then $x\notin U_{m_0}$ for some $m_0\in\N$, so $x\notin U_m$ for all $m>m_0$ because $(U_m)_{m=1}^{\infty}$ is decreasing. This fact and the inclusion $S\subseteq U_m$ for all $m\in\N$ imply that both sides of  $\lim_{m\to \infty}1_{U_m}(x)=1_S(x)$ equal 0. But $(\cap_{m=1}^{\infty}U_m)\backslash S$ is a measurable subset satisfying  $\mu((\cap_{m=1}^{\infty}U_m)\backslash S)\leq \mu(U_m\backslash S)<1/m$ for each $m\in\N$. Thus $\mu((\cap_{m=1}^{\infty}U_m)\backslash S)=0$. It follows that  $\lim_{m\to \infty}1_{U_m}(x)=1_S(x)$ and hence $\lim_{m\to \infty}(1_{U_m}(x)e^{ig(x)})=1_S(x)e^{ig(x)}$ for almost all $x\in \R^n$. Since $(U_m)_{m=1}^{\infty}$ is a decreasing sequence of subsets which contain $S$ we have $1_S(x)=|1_S(x)e^{ig(x)}|\leq |1_{U_m}(x)e^{ig(x)}|=1_{U_m}(x)\leq 1_{U_1}(x)$ for all $x\in \R^n$ and all $m\in\N$. Since $1_{U_1}$ is Lebesgue integrable on $\R^n$ (because $\mu(U_1)\leq \mu(U_1\backslash S)+\mu(S)<1+\mu(I_0)<\infty$), we can conclude the following equality from the dominated convergence theorem and the fact established earlier that the integral of $e^{ig}$ vanishes on every nonempty open subset of $\R^n$:
\begin{equation*}
\int_{S}e^{ig(w)}dw=\int_{\R^n} 1_S(w)e^{ig(w)}dw=\lim_{m\to\infty}\int_{\R^n}1_{U_m}e^{ig(w)}dw=\lim_{m\to\infty}\int_{U_m}e^{ig(w)}dw=0. 
\end{equation*}
Because $S$ was an arbitrary measurable subset of $I_0$ we can conclude that $e^{ig(w)}=0$ for  almost every $w\in I_0$ (see, e.g., \cite[p. 105]{KurtzSwartz2004}). On the other hand, we obviously have $|e^{ig(w)}|=1$  for every $w\in I_0$ since the range of $g$ is $\R$. This contradiction shows the existence of a positive rational $\alpha$ such that $\int_{I_0} e^{i\alpha g(x)}dx\neq 0$, as claimed.  
\end{proof}
In Lemma \bref{lem:UniqueRn} above instead of a hypercube we can take any hyper-parallelepiped spanned by some basis (the proof is the same), but we have found no real value in working with  a general basis  since any hyper-parallelepiped contains a hypercube and we are interested only in local measurability of $e^{ig}$. On the other hand, in the abstract regularity condition mentioned in Remark \bref{rem:Abstract} below, we preferred to use a more general statement using a general basis. 

\begin{proof}[{\bf Proof of Theorem \bref{thm:CauchyMeasurable}}]
Let $I$ be a hypercube on which $e^{if}$ is measurable and suppose that $U:=\{u_1,\ldots,u_n\}$ is the (orthogonal) basis of $\R^n$ which spans $I$. Consider the $n\times n$ matrix whose $k$-th raw is $u_k$, $k\in\{1,\ldots,n\}$. This matrix is invertible since $\{u_1,\ldots,u_n\}$ is a basis. Hence there exists a unique  $c=(c_k)_{k=1}^n\in \R^n$ which satisfies the linear system 
\begin{equation}\label{eq:uc=f(u)}
u_k\cdot c=f(u_k),\quad k\in\{1,\ldots,n\}.
\end{equation}
One solution to \beqref{eq:Cauchy} is $g_1:=f$ and another one is the linear function 
$g_2(x):=c\cdot x$, $x\in\R^n$. Moreover, by the choice of $c$ we have $g_1(u_k)=g_2(u_k)$ for all $k\in\{1,\ldots,n\}$. This fact and the linearity of \beqref{eq:Cauchy} imply that $g:=g_1-g_2$ satisfies the  following initial value problem 
\begin{subequations}\label{eq:CauchyThm_0}
\begin{equation}\label{eq:CauchyThm_g}
g(x+y)=g(x)+g(y)\quad \forall x,y\in \R^n, 
\end{equation}
\begin{equation}\label{eq:periodThm_g}
g(u_k)=0,\quad k=1,\ldots,n.
\end{equation}
\end{subequations}
 We claim that \beqref{eq:CauchyThm_0} has $g\equiv 0$ as its unique solution in the class of functions $g:\R^n\to\R$ having the property that $e^{ig}$ is measurable on $I$. Once we prove this claim we can conclude that $g_1=g_2$ because $g:=g_1-g_2$ solves \beqref{eq:CauchyThm_0} and it satisfies $e^{ig(x)}=e^{if(x)}e^{-ic\cdot x}$ for all $x\in\R^n$, namely it is a multiplication of two functions which are measurable on $I$ and hence it is also measurable on $I$. To see that the above mentioned claim holds, we observe that the zero function obviously solves  \beqref{eq:CauchyThm_0} and it belongs to the considered class of functions. On the other hand, given any solution $g$ to \beqref{eq:CauchyThm_0} in the considered class of functions, we have $g(x+u_k)=g(x)+g(u_k)=g(x)$ for all $x\in\R^n$ and all $k\in\{1,\ldots,n\}$, that is, $g$ is $U$-periodic. Since $e^{ig}$ is measurable on $I$ we can use Lemma \bref{lem:UniqueRn} to deduce that indeed $g\equiv 0$, i.e., we also have uniqueness, as required.
\end{proof}

\begin{remark}\label{rem:Integral}
Under the stronger assumption that $f:\R^n\to \R$ satisfies \beqref{eq:Cauchy} for every $x,y\in\R^n$ and is locally integrable, the approach used in this section allows one to give a simple proof that $f$ 
must be linear. Indeed, since $f$ is locally integrable, there exists a hypercube $I$ on which $f$ is integrable. Let $U:=\{u_1,\ldots,u_n\}$ be the basis which spans $I$. The linearity of \beqref{eq:Cauchy} implies, as in the second proof of Theorem \bref{thm:CauchyMeasurable} given above, that it is enough to show that $g\equiv 0$ 
is the unique solution to \beqref{eq:CauchyThm_0} in the class of functions which are integrable on $I$. By substituting $y:=u_k$  in \beqref{eq:Cauchy} we see that $g(x+u_k)=g(x)+0$ for all $x\in \R^n$ and all $k\in \{1,\ldots,n\}$, i.e., $g$ is $U$-periodic. Since we assume that $g$ is integrable on $I$, its periodicity implies (by Lemma \bref{lem:period2}) that it is integrable on any translated copy of $I$. Now, by fixing $y\in \R^n$ in \beqref{eq:Cauchy} and integrating over $I$ with respect to $x$, we see that 
\begin{equation*}\label{eq:Int}
\int_{I}g(x+y)dx=\int_I g(x)dx+g(y)\int_I dx. 
\end{equation*}
But $\int_{I}g(x+y)dx=\int_I g(x)dx$ according to Lemma \bref{lem:period2}. Hence $g(y)\mu(I)=0$ for each $y\in \R$, that is, $g\equiv 0$.
\end{remark}

\begin{remark}\label{rem:Nonlinear}
The nonlinear solutions of \beqref{eq:Cauchy} are examples of exotic functions, not only because they are not measurable 
(or, as Theorem \bref{thm:CauchyMeasurable} shows, even a complex exponent of them cannot be measurable), but, for instance, because their graphs are dense in $\R^2\,\,$  \cite[p. 14]{AczelDhombres},\cite{Hamel},\cite{Wilansky}. (But see, e.g., \cite{Jones1942,Wilansky} for ``nice'' properties that some of them satisfy.) The approach presented in 
this section suggests an idea  for creating functions 
which are even more exotic. 

Indeed, assume for simplicity that $f:\R\to\R$ and let $H:=\{x_{\gamma}: 
 \gamma\in \Gamma\}$ be a Hamel basis, that is, a basis of $\R$ where $\R$ is considered as a vector space over the field of rationals $\Q$. It is well-known that the  cardinality of $H$ is the continuum cardinality, that is, the same as the cardinality of $\R$ (see \cite[Theorem 4.2.3, p. 82]{Kuczma2009book}). As is well-known, and was observed by Hamel \cite{Hamel} (see also \cite[pp. 35--36]{Aczel},\cite[pp. 18--19]{AczelDhombres}), the general solution $f:\R\to\R$ which satisfies \beqref{eq:Cauchy} for all $x,y\in\R$ is obtained by associating with each $\gamma\in\Gamma$ an arbitrary $y_{\gamma}\in \R$, setting  $f(x_{\gamma}):=y_{\gamma}$, and extending $f$ to $\R$ in a linear manner (that is, given $y\in \R$, since $y=\sum_{k=1}^m q_{k} x_{{\gamma}_k}$ for some $m\in\N$, $q_1,\ldots,q_m\in\Q$, and $x_{{\gamma}_1},\ldots, x_{{\gamma}_m}\in H$, we define  $f(y):=\sum_{k=1}^m q_{k} y_{{\gamma}_k}$). Now fix an arbitrary nonempty finite or countable set $J$  of basis elements, define $f(x_{\gamma})$ to be any nonzero real number for all $\gamma\in J$, define  $f(x_{\gamma}):=0$ for the other basis elements $x_{\gamma},\,\gamma\notin J$, and extend $f$ to $\R$ in a linear manner as  explained above (in particular,  when $\card(J)=\aleph_0$, then there are at least $2^{\aleph_0}=\card(\R)$ such functions $f$). The exotic property that $f$ has is that it is  a non-constant solution to  \beqref{eq:Cauchy} which has a set of  linearly independent (over $\Q$) periods whose cardinality is $\card(\R)$. Indeed, the equality $f(x_{\gamma})=0$ for each $\gamma\in \Gamma\backslash J$ implies that   $f(x+x_{\gamma})=f(x)+f(x_{\gamma})=f(x)$ for each $x\in\R$, i.e., $x_{\gamma}$ is a period of $f$ for every $\gamma\in \Gamma\backslash J$, and it remains to observe that  $\card(\Gamma\backslash J)=\card(\Gamma)=\card(H)=\card(\R)$. 

Actually, it is possible to make the set of periods of $f$ to be dense in $\R$ (this, in  particular, will also show that $f$ is so-called microperiodic \cite[p. 332]{Kuczma2009book}, that is, it has arbitrary small positive  periods). Indeed, since the set of periods of $f$ coincides with the Hamel basis that we fixed in the beginning, it is sufficient to show that there exists a Hamel basis which is dense in $\R$ and then to define $f$ with respect to that  basis as done above. The existence of a dense Hamel basis is known (the main arguments below are based on ideas of Zoltan Boros \cite{Boros2015personal}): roughly speaking, one starts with an arbitrary Hamel basis $H$ and then rescales its elements by appropriate rational coefficients so the resulting basis will be dense. More precisely, one first observes using the Cantor-Schroeder-Bernstein theorem that the cardinality of the set $\mathscr{K}:=\{(a,b): a,b\in\R,\, a<b\}$ of all bounded and open intervals in $\R$ is at most $\card(\R^2)=\card(\R)$ and at least $\card(\R)$, and hence it is $\card(\R)$. Therefore there exists a bijection $\psi$ from $\mathscr{K}$ to $H$. Now, given $I\in \mathscr{K}$, since $\psi(I)\in H$ we have $\psi(I)\neq 0$ because the elements of a Hamel basis do not vanish. Since for all $x\in\R$, $x\neq 0$ the set $\{qx: q\in\Q\backslash\{0\}\}$ is dense in $\R$, it follows that $\{q\psi(I): q\in \Q\backslash\{0\}\}$ is dense in $\R$. From this result and the fact that $I$ is an open interval it follows that there exists a rational number $q_I$ such that $q_I\psi(I)\in I$. Denote $\tilde{H}:=\{q_I\psi(I): I\in\mathscr{K}\}$. Then $\tilde{H}$ is a weighted rescaling of $H$ by non-zero rationals and hence $\tilde{H}$ is also a Hamel basis. Since from the construction of $\tilde{H}$ for all $I\in \mathscr{K}$ we can find an element of $\tilde{H}$ which belongs to $I$, it follows that $\tilde{H}$ is dense in $\R$, as required. 
\end{remark}

\section{The case where $e^{if}$ is measurable, topological torus}\label{sec:Torus}
Now $f$ is assumed to be a real function defined on the $n$-dimensional topological (flat) torus $\T^n:=\R^n/\Z^n$ and having the property that $e^{if}$ is (Haar) measurable. We need the following known lemma which is a variation of Lemma \bref{lem:period2}. For the sake of completeness and since it is not very easy to find a full proof of this Lemma in the literature, we decided to include a proof in the appendix (Section \bref{sec:Appendix} below). 
\begin{lem}\label{lem:Torus}
Given any $y\in \T^n$ and $h:\T^n\to \C$, if $h\in L_1(\T^n)$, then the translation function $h_y(x):=h(x+y)$, $x\in\T^n$ is (Haar) integrable and we have 
$\int_{\T^n}h_y(x)d\mu(x)=\int_{\T^n}h(x)d\mu(x)$. 
\end{lem}

\begin{thm}\label{thm:CauchyTorus}
Given an additive  homomorphism $f$ from the finite dimensional topological torus $\T^n$ to $\R$, if $e^{if}$ is (Haar) measurable, then $f\equiv 0$.
\end{thm}
\begin{proof}
The proof is almost word for word as in the proof of Lemma \bref{lem:UniqueRn}. Here we use the periodicity of $f$ with respect to the  basis which spans the torus, and Lemma \bref{lem:Torus} instead of  Lemma \bref{lem:period2}. We note that we need the change of variables formula for scaling only for small enough (at most 1) positive rational $\alpha$, and in this case this formula is true for the topological torus.
\end{proof}

\section{Jensen's equation}\label{sec:Jensen}
This is the functional equation
\begin{equation}\label{eq:Jensen}
f\left(\frac{x+y}{2}\right)=\frac{f(x)+f(y)}{2}
\end{equation}
where $x$ and $y$ belong to, say, $\R^n$ or a subset of $\R^n$ and $f$ is a real-valued function.  
See \cite[pp. 43--48]{Aczel},\cite[pp. 242--243]{AczelDhombres},\cite[pp. 351--354]{Kuczma2009book} for more 
details and related references. 
The next theorem extends \cite[Theorem 1, p. 46]{Aczel},\cite[Proposition 3, p. 243]{AczelDhombres},\cite[Theorem 13.2.3, p. 354]{Kuczma2009book}.
\begin{thm}\label{thm:Jensen}
Let $S$ be a convex subset of $\R^n$ and assume that its interior is nonempty. 
Suppose that $f:S\to\R$ satisfies \beqref{eq:Jensen} for all $(x,y)\in S^2$.  
If $e^{if}$ is measurable on $S$, then there exist $c\in\R^n$ and $b\in\R$ 
such that $f(x)=c\cdot x+b$ for each $x\in S$. 
\end{thm}
\begin{proof}
Because $S$ is convex, $(x+y)/2\in S$, so \beqref{eq:Cauchy} is well defined. By \cite[Theorem 13.2.1, p. 353]{Kuczma2009book} there exists a function $g:\R^n\to\R$ satisfying \beqref{eq:Cauchy} for every $x,y\in\R^n$ and a constant $b\in \R$ such that $f(x)=g(x)+b$ for all $x\in S$. 
Since $e^{if}$ is measurable on $S$ and the interior of $S$ is nonempty, 
$e^{if}$ is measurable on a hypercube contained in $S$. Therefore $e^{ig}=e^{-ib}e^{if}$ 
is measurable on this hypercube, and by Theorem \bref{thm:CauchyMeasurable} 
there exists $c\in\R^n$ such that $g(x)=c\cdot x$ for each $x\in\R^n$. 
Thus $f(x)=c\cdot x+b$ for each $x\in S$, and an immediate check shows that 
this function does satisfy \beqref{eq:Jensen} for all $(x,y)\in S^2$. 
\end{proof}

\section{Cauchy's equation on restricted domains}\label{sec:Restricted}
Now we consider the Cauchy functional equation \beqref{eq:Cauchy} assuming $f$ is real-valued function defined on a subset 
of $\R^n$ (and hence $(x,y)$ should belong to a subset of $\R^{2n}$). 
Sometimes \beqref{eq:Cauchy} with these additional conditions is also 
referred to as ``conditional Cauchy's equation'' or ``the restricted Cauchy equation''. 
 More details about this (quite studied) issue can be found 
 in \cite[pp. 44--49]{Aczel},\cite[pp. 12, 18, 73, 84--92]{AczelDhombres}, 
 \cite{Dhombres1979},\cite[pp. 369--374]{Kuczma2009book}. 
 
In what follows we will derive a few theorems. The first one extends \cite[Corollary 9, p. 18]{AczelDhombres} in which $S=[0,\infty)$. It is based on a certain definition (see Remark \bref{rem:RestrictedCauchy} below for related examples) and several lemmas  
 (the first lemma generalizes the discussion in \cite[p. 12]{AczelDhombres} and is also related 
 to \cite{ABDKR1971}). 
  
\begin{defin}\label{def:StronglySpan}
Let $(G,+)$ be a group and let $S\subseteq G$. 
\begin{enumerate}[(a)]
\item $S$ is said to subtractively span $G$ 
if $G=S-S:=\{s-t: (s,t)\in S^2\}$. 
\item $S$ is said to strongly subtractively span $G$ if for 
all $x_1,x_2\in G$ there exist $s_1,s_2,t_1,t_2\in S$ 
such that $x_i=s_i-t_i$, $i=1,2$ and such that $s_1+s_2\in S$, $t_1+t_2\in S$.
\end{enumerate}
\end{defin}

\begin{lem}\label{lem:ExtensionHomomorphism}
Let $(G,+)$ and $(H,+)$ be two commutative groups. Suppose that 
$S\subseteq G$ strongly subtractively spans $G$. Let $S+S:=\{a+b: (a,b)\in S^2\}$. 
Let $A\subseteq G$ satisfy $S\cup(S+S)\subseteq A$. 
If $f:A\to H$ satisfies 
\begin{equation}\label{eq:CauchyRestricted}
f(x+y)=f(x)+f(y),\quad \forall (x,y)\in S^2,
\end{equation}
then there exists an additive homomorphism $F:G\to H$ which coincides with $f$ on $S$. 
\end{lem}

\begin{proof}
Since $S$ strongly subtractively spans $G$  it subtractively spans it as an immediate verification shows.
Hence, given $x\in G$, there exists $(s,t)\in S^2$ such that $x=s-t$. Define 
\begin{equation*}
F(x):=f(s)-f(t).
\end{equation*}
To see that $F$ is well defined, suppose that we also have $x=s'-t'$ for another pair  $(s',t')\in S^2$. Then 
$s+t'=s'+t$ and $f(s)+f(t')=f(s+t')=f(s'+t)=f(s')+f(t)$ because of \beqref{eq:CauchyRestricted}. 
It follows that $f(s)-f(t)=f(s')-f(t')$ and hence $F$ is indeed well defined. 
To see that $F$ coincides with $f$ on $S$ let $x\in S$ be given. As explained above, because $S$ 
subtractively spans $G$ there exists $(s,t)\in S^2$ such that $x=s-t$. 
By \beqref{eq:CauchyRestricted} it follows that $f(s)=f(x+t)=f(x)+f(t)$. 
This and the definition of $F$ imply the equality $F(x)=f(s)-f(t)=f(x)$. Finally, 
to see that $F$ is additive let $(x_1,x_2)\in G^2$ be arbitrary. 
Since $S$ strongly subtractively spans $G$ there exist $s_1,s_2,t_1,t_2,\in S$ which 
satisfy the relations  $x_1=s_1-t_1$, $x_2=s_2-t_2$, $s_1+s_2\in S, t_1+t_2\in S$. 
This and the commutativity of $G$ imply that $x_1+x_2=(s_1+s_2)-(t_1+t_2)\in S-S$. 
This, the definition of $F$, the commutativity of $H$, and \beqref{eq:CauchyRestricted} imply the required assertion: 
\begin{equation*}
F(x_1+x_2)=f(s_1+s_2)-f(t_1+t_2)=f(s_1)-f(t_1)+f(s_2)-f(t_2)=F(x_1)+F(x_2). 
\end{equation*}
\end{proof}

\begin{thm}\label{thm:RestrictedCauchy}
Suppose that $S\subseteq \R^n$ strongly subtractively spans $\R^n$. 
Let $A\subseteq \R^n$ satisfy $S\cup(S+S)\subseteq A$. 
Let $f:A\to\R$ and assume that $f$ satisfies \beqref{eq:CauchyRestricted}.
If $S$ contains a hypercube $I$ on which   $e^{if}$ is measurable, then 
there exists $c\in\R^n$ such that $f(x)=c\cdot x$ for each $x\in S$. 
\end{thm}
\begin{proof}
By Lemma \bref{lem:ExtensionHomomorphism} there exists an additive function 
$F:\R^n\to\R$ satisfying $F(x)=f(x)$ for all $x\in S$. 
Since $F$ is additive and $e^{iF}=e^{if}$ is measurable on $I$, it follows from 
Theorem \bref{thm:CauchyMeasurable} that there exists $c\in \R^n$ such 
that $F(x)=c\cdot x$ for all $x\in \R^n$. The assertion follows since 
$f(x)=F(x)$ for all $x\in S$. 
\end{proof}

\begin{remark}\label{rem:RestrictedCauchy}
Examples of subsets $S$ of $\R^n$ which strongly subtractively span $\R^n$ and containing a  
hypercube are orthants (with or without the boundary), 
halfspaces, and any set containing a subset which strongly subtractively spans $\R^n$ and has a nonempty 
interior. 
In addition, any semigroup $S\subseteq \R^n$ (that is, $S+S\subseteq S$) 
which subtractively spans $\R^n$  strongly subtractively spans $\R^n$. 
In particular, $S:=\bigcup_{m=1}^{\infty}[2m,3m]$ 
strongly subtractively spans $\R$. There are also examples of subsets $S$ with 
a nonempty interior which strongly subtractively span $\R^n$ but are not semigroups. A simple example is a translated copy of an orthant by a vector 
which does not belong to the orthant. But there are more exotic examples such as $S:=\bigcup_{m=1}^{\infty}S_m$ 
where $S_m:=[10^m,5\cdot 10^m)^n$ for each $m\in\N$ (because given $x,y\in \R^n$, let $2\leq m\in\N$ 
be such that $\|x\|+\|y\|<10^{m-1}$; let $p_m:=(2\cdot 10^{m})_{k=1}^n$ be the 
 vector whose components are $2\cdot 10^{m}$; 
 then $x=(p_m+x)-p_m\in S_m-S_m$, $y=(p_m+y)-p_m\in S_m-S_m$, and  $x+y=(2p_m+x+y)-2p_m\in S_m-S_m$ 
 as claimed). 
\end{remark}

The same reasoning of the proof of Theorem \bref{thm:RestrictedCauchy} 
can be used for easily extending it much further, as shown in the next 
theorem. 
\begin{thm}
Let $S\subseteq \R^n$. Let $A\subseteq \R^n$ satisfy $S\cup(S+S)\subseteq A$. 
Suppose that $f:A\to\R$ satisfies \beqref{eq:CauchyRestricted}. 
Suppose also that $S$ contains a hypercube $I$. If $e^{if}$ is measurable on $I$ and 
if there exists a (not necessarily unique) additive function $F:\R^n\to\R$ 
such that $F(x)=f(x)$ for all $x\in S$, then  there exists $c\in\R^n$ such that 
$f(x)=c\cdot x$ for each $x\in S$. 
\end{thm}
Now the theory developed in the literature can be used, so for instance, 
we may take $S$ to be  any interval in $\R$ having 0 in 
its closure \cite[Theorem 13.5.2, p. 368]{Kuczma2009book} and $A:=S+S$. We can also take $S$ to be a multidimensional ball having the origin as its center \cite[Corollary 13.6.3, p.  373]{Kuczma2009book} 
(and $A:=S+S$) but this actually follows from the next theorem (which generalizes 
the case of a continuous additive function defined on a compact interval \cite[pp. 45--46]{Aczel}). 
\begin{thm}\label{thm:RestrictedConvex}
Let $S\subseteq \R^n$ be a convex subset having a nonempty interior. 
Let $A\subseteq \R^n$ satisfy $S\cup(S+S)\subseteq A$. 
Suppose that $f:A\to\R$ satisfies \beqref{eq:CauchyRestricted}. 
 If $e^{if}$ is measurable on $S$, 
 then  there exists $c\in\R^n$ such that $f(x)=c\cdot x$ for each $x\in S$.
\end{thm}
\begin{proof}
Let $u\in S+S$ be arbitrary. Then $u=s+t$ for some $s,t\in S$ 
and hence $u/2=(s+t)/2\in S$ because $S$ is convex. By putting $x=y=u/2$ 
in \beqref{eq:CauchyRestricted} it follows that $f(u)/2=f(u/2)$ for 
all $u\in S+S$. In particular this is true for $u:=x+y$ where $x,y\in S$ are given. 
Thus $f((x+y)/2)=f(x+y)/2=(f(x)+f(y))/2$ where the last equality follows from \beqref{eq:CauchyRestricted}. 
Hence $f$ satisfies the Jensen equation \beqref{eq:Jensen} for all $(x,y)\in S^2$, 
and from Theorem \bref{thm:Jensen} it follows that there exist $c\in \R^n$ 
and $b\in\R$ such that $f(x)=c\cdot x+b$ 
for each $x\in S$. By plugging this expression in \beqref{eq:CauchyRestricted} 
it follows that $b=0$. 
\end{proof}

\section{The multiplicative Cauchy equation}\label{sec:Multiplicative}
This is the functional equation 
\begin{equation}\label{eq:CauchyMult}
f(x+y)=f(x)f(y) 
\end{equation}
where the pairs $(x,y)$ belong to a subset of $\R^{2n}$ and $f$ is a real-valued function. 
See \cite[pp. 37--39]{Aczel},\cite[pp. 28--29]{AczelDhombres},\cite[pp. 343-344, 349--350]{Kuczma2009book} for related 
details and related references. 
An obvious solution to \beqref{eq:CauchyMult} is $f\equiv 0$. 
As is well known, when the whole space is considered, if $f(x_0)=0$ at some point $x_0\in\R^n$, 
then $f(x)=f(x-x_0)f(x_0)=0$ for each $x\in\R^n$. Since in addition $f(x)=(f(x/2))^2$ for each $x\in \R^n$, 
it follows that when $f$ is not the constant zero function, then $f(x)>0$ for all $x$. 
However, when the pairs $(x,y)$ belong only to a subset of $\R^{2n}$, then  
it is not clear why a solution to \beqref{eq:CauchyMult} must be positive. Hence, in the next 
theorem (which extends  \cite[Theorem 1, pp. 38--39]{Aczel},
\cite[Theorem 5, p. 29]{AczelDhombres},
\cite[Theorem 13.1.4, p. 349]{Kuczma2009book}, \cite[Theorem 13.1.7, p. 350]{Kuczma2009book}) 
the positivity of $f$ is assumed in advance. 
\begin{thm}
Let $S\subseteq \R^n$. Let $A\subseteq \R^n$ satisfy $S\cup(S+S)\subseteq A$. 
Suppose that $f:A\to\R$ is a positive function satisfying \beqref{eq:CauchyMult} for all $(x,y)\in S^2$. 
Assume that either $S$ is a convex subset having a nonempty interior and $f^i$ is measurable on $S$, 
or that $S$ strongly subtractively spans $\R^n$ and it contains a hypercube $I$ on which $f^i$ is 
measurable. Then there exists $c\in\R^n$ such that $f(x)=e^{c\cdot x}$ for each $x\in S$. 
\end{thm}
\begin{proof}
Since $f$ is positive, the function $g:A\to\R$ defined by $g:=\ln(f)$ is well-defined on $A$ and by taking logarithm on 
\beqref{eq:CauchyMult} we see that $g$ is additive. Since $f^i=\exp(i\ln(f))=e^{ig}$ is assumed 
to be measurable on either $S$ or $I$, Theorem \bref{thm:RestrictedConvex} or Theorem \bref{thm:RestrictedCauchy} respectively imply  the existence of $c\in\R^n$ 
such that $g(x)=c\cdot x$ for each $x\in \R^n$ and hence $f(x)=e^{c\cdot x}$ for all $x\in S$. This $f$ indeed 
satisfies \beqref{eq:CauchyMult} and the assertion follows. 
\end{proof}

\section{An alternative Cauchy's equation}\label{sec:Alternative}
This is the functional equation
\begin{equation}\label{eq:Alternative}
(f(x+y))^2=(f(x)+f(y))^2 
\end{equation}
where $x,y\in\R^n$ are arbitrary and $f$ is a real-valued function. 
It seems to appear first in the context of certain problems related to physics \cite{Hosszu1964}. 
See  \cite{FischerMuszely1967},\cite{Hosszu1963},\cite{Kuczma1978},\cite[pp. 380--382]{Kuczma2009book} 
for more details and related references. 
Obviously, any solution to the Cauchy functional equation \beqref{eq:Cauchy} 
solves \beqref{eq:Alternative}, but the converse is less trivial 
because at least at first glance one may only deduce the relation 
 $f(x+y)=\pm(f(x)+f(y))$ (where the sign depends on the pair $(x,y)$). However, as follows from  \cite[Theorem 13.9.2, p. 382]{Kuczma2009book},  
 any solution to  \beqref{eq:Alternative} defined on a semigroup $S$ of $\R^n$ must be additive. 
This, Theorem \bref{thm:RestrictedCauchy}, and Remark \bref{rem:RestrictedCauchy}, 
imply the following theorem.
\begin{thm}
Suppose that $S$ is a semigroup of $\R^n$ which subtractively spans $\R^n$. Assume that 
$f:S\to\R$ satisfies \beqref{eq:Alternative} for every $x,y\in S$. If there exists a hypercube $I\subset S$ on which 
$e^{if}$ is measurable, then there exist $c\in\R^n$ 
such that $f(x)=c\cdot x$ for each $x\in S$. 
\end{thm}

\section{Pexider's equation}\label{sec:Pexider}
This is a functional equation generalizing the Cauchy functional equation and 
involving three unknown functions:
\begin{equation}\label{eq:Pexider}
f(x+y)=g(x)+h(y) 
\end{equation}
where $(x,y)$ belongs to  $\R^{2n}$ or a subset of $\R^{2n}$ and $f$,$g$, and $h$ are all real-valued functions. 
See \cite[pp. 141--142]{Aczel},\cite[pp. 42--43]{AczelDhombres},\cite[pp. 355--363]{Kuczma2009book},\cite{Pexider1903}
for more details and additional references. 

The next theorem extends well-known results regarding the solvability of \beqref{eq:Pexider}  (see, e.g., \cite[Corollary, p. 142]{Aczel}, \cite[Theorem 13.3.9, p. 363]{Kuczma2009book}). 
\begin{thm}\label{thm:Pexider}
Let $S\subseteq \R^n$ be a semigroup satisfying $0\in S$. Assume that $f:S\to\R$, $g:S\to\R$, $h:S\to\R$ 
satisfy \beqref{eq:Pexider} for all $(x,y)\in S^2$. Suppose that $S$ subtractively spans 
$\R^n$ and it contains a hypercube $I$ on which a complex exponent of one of 
the given functions is measurable. 
Then there exist $c\in \R^n$ and constants $a,b\in\R$ such that $f(x)=c\cdot x+a+b$, $g(x)=c\cdot x+a$, 
and $h(x)=c\cdot x+b$ for all $x\in S$.
\end{thm}
\begin{proof}
Let $a:=g(0)$, $b:=h(0)$ and let $p:S\to\R$ be defined by $p(x):=f(x)-a-b$ for all $x\in S$. 
An immediate verification shows that 
$f(x)=p(x)+a+b$, $g(x)=p(x)+a$, $h(x)=p(x)+b$ for all $x\in S$ 
and that $p$ is additive. 
Since a complex exponent of one of the functions $f$, $g$, or $h$ is assumed to 
be measurable on a certain hypercube, the above expressions imply that $e^{ip}$ 
is measurable on this hypercube. Since a semigroup which subtractively spans a commutative group actually strongly 
subtractively spans it (this is an immediate consequence of Definition \bref{def:StronglySpan}) and since $S+S\subseteq S$, Theorem \bref{thm:RestrictedCauchy} (with $A:=S$) implies that there exists $c\in\R^n$ 
such that $p(x)=c\cdot x$ 
for each $x\in S$. The assertion follows after easily checking that the 
obtained triplet $(f,g,h)$ does solve \beqref{eq:Pexider}. 
\end{proof}

\section{Stability}\label{sec:Stability} Consider the following approximate version of \beqref{eq:Cauchy}: 
\begin{equation}\label{eq:Stability}
|f(x+y)-f(x)-f(y)|\leq\epsilon 
\end{equation}
for all $x,y\in \R^n$ where $f$ is a real-valued function. Here $\epsilon$ is a given positive number. A function 
$f:\R^n\to\R$ satisfying \beqref{eq:Stability} is called an approximately additive function 
or an $\epsilon$-additive function. Inequality \beqref{eq:Stability} is a perturbed 
version of \beqref{eq:Cauchy} and one may be interested in knowing whether any 
$\epsilon$-additive function is a small perturbation of a pure additive function. 
This issue, which was raised by Ulam in 1940 \cite[p. 222]{Hyers1941}, \cite[p. 64]{Ulam1960} (and traces of it can be found in the 1925 book of P\'olya and Szeg\"o \cite[Problem 99, p. 17]{PolyaSzego1925book}), is, in some sense, related to the question 
of the stability of solutions to differential equations. Anyhow, it was proved by 
Hyers \cite{Hyers1941} 
 that the answer to the above question is positive: there exists 
 a unique function $g:\R^n\to\R$ satisfying \beqref{eq:Cauchy} 
and $|f(x)-g(x)|\leq\epsilon$ for all $x\in \R^n$. 
The additive function $g$ satisfies 
\begin{equation}\label{eq:Lim_g_fm}
g(x)=\lim_{m\to\infty}\frac{f(mx)}{m} 
\end{equation}
for all $x\in \R^n$ (a proof can be found  in  \cite[p. 484]{Kuczma2009book}; the 
more common expression is $g(x)=\lim_{m\to\infty}f(2^m x)/2^m$). 
As a matter of fact, Hyers proved his theorem for functions acting between Banach spaces. As is well-known, 
 and follows e.g., from the proof of  \cite[Theorem 17.1.1, pp. 483--484]{Kuczma2009book} 
(which is a minor variation of Hyers' proof), the above holds 
whenever $f:S\to X$ satisfies \beqref{eq:Stability} for all $x,y\in S$ where $X$ is a Banach space and $S$ is a semigroup. 

The issue of stability of functional equations (and in particular, of the Cauchy 
functional equation) and various problems related 
to it have become quite popular during the 
last decades. This is somewhat illustrated in the following very short (and semi random) list of related works: \cite{Aoki1950,ArriolaBeyer2005-6,Badora2000,BGP2003jour,Bourgin1951,Chung2012,DalesMoslehian2007jour,Gajda1988,Ger1993jour,
GerSikorska1997,MaligrandaAM2008,MihetRadu2008jour,MirmostafaeeMoslehian2009jour,MoslehianRassias2007jour,NajatiRahimi2008,Paneah2009,Rhassias1978,Tabors2008}. 
A few reviews related to the issue of stability are  \cite{Forti1995,HIR1998,HyersRassias1992,Szekelyhidi2000incol}. 
Here  we will concentrate only on extending \cite[Theorem 17.1.2, p. 485]{Kuczma2009book}, saying that under 
familiar regularity conditions (e.g., $f$ is measurable or bounded above on a set 
of positive measure) the associated additive function $g$ must be continuous. In a sense, this result 
is related to the stability notion involving integrability conditions mentioned in \cite{ABR1995,Elliott1984} 
because of the measurability condition which is involved. However, the results and proofs mentioned 
there (in particular, the stability condition on $f$) are different from the one given below and also the setting is different (the functions are from the positive real line to itself). 

\begin{thm}\label{thm:Stability}
Let $S\subseteq \R^n$ be a semigroup which subtractively spans $\R^n$ and containing 
a hypercube $I$. If $f:S\to\R$ satisfies \beqref{eq:Stability} for all $x,y\in S$ and if there exists an infinite sequence $(m_k)_{k=1}^{\infty}$ of natural numbers such that each of the functions  $h_k:S\to\C$ defined by $h_k(x):=e^{if(m_kx)/m_k}$ for every $x\in S$  is measurable on $I$, then there 
exists $c\in\R^n$ such that $|f(x)-c\cdot x|\leq \epsilon$ for all $x\in S$.  
\end{thm}
\begin{proof}
 As mentioned near \beqref{eq:Lim_g_fm}, there exists an additive function $g:S\to\R$ satisfying  $|f(x)-g(x)|\leq \epsilon$ for all $x\in S$. 
From \beqref{eq:Lim_g_fm}  we have $g(x)=\lim_{k\to\infty}f(m_k x)/m_k$ 
for all $k\in \N$ and $x\in S$. Of course, $f(mx)$ is well-defined for all $x\in S$ and $m\in \N$ because $S+S\subseteq S$. By the continuity of the exponential function we have 
\begin{equation*}
h(x):=\exp(ig(x))=\lim_{k\to\infty}\exp(if(m_k x)/m_k)=\lim_{k\to\infty}h_k(x)  
\end{equation*}
for all $x\in S$. Thus the restriction of $h$ to $I$ is a pointwise limit of measurable 
functions and hence measurable \cite[p. 15]{Rudin1987book}. Since $g$ satisfies \beqref{eq:Cauchy} for every $x,y\in S$  
 and since a semigroup which subtractively spans a commutative group actually strongly 
subtractively spans it (as follows from Definition \bref{def:StronglySpan}), we conclude from 
Theorem \bref{thm:RestrictedCauchy} (with $A:=S$ and with $g$ instead of $f$) the existence of $c\in\R^n$ such that $g(x)=c\cdot x$ for all $x\in S$ 
 and the assertion follows.  
\end{proof}

\begin{remark}\label{rem:exp_fm}
If $f$ is locally measurable, then $x\mapsto e^{if(mx)/m}$ is locally measurable for all $m\in\N$ and hence the condition on the functions $h_k$ from Theorem \bref{thm:Stability} is satisfied and therefore $g$ is measurable. On the other hand, if we only know that $e^{if}$ is measurable, then it is not necessarily true that $g$ is measurable. Indeed, consider as in Remark \bref{rem:e^if} 
the function $f(x):=2\pi[g(x)]$, $x\in\R$ where $g$ is an arbitrary nonlinear solution of \beqref{eq:Cauchy}. 
Then $e^{if}$ is even continuous since it equals the constant function $1$. Since $|t-[t]|\leq 1$ for all $t\in\R$ we have $|f(x)-(2\pi)g(x)|\leq 2\pi$ for each $x\in \R$. This inequality, the  triangle inequality, and the additivity of $g$,  imply that $f$ satisfies \beqref{eq:Stability} for every $x,y\in\R$ with $\epsilon:=6\pi$. Since $2\pi g$ satisfies \beqref{eq:Cauchy} for every $x,y\in\R$, 
Hyers' theorem implies that this is the unique additive function which $\epsilon$-approximates $f$ 
and also that $2\pi g(x)=\lim_{m\to\infty}f(mx)/m$ for all $x\in \R$ (this limit can also be 
 computed directly). But $g$ is non-measurable. We can actually say more:  
although $e^{if}$ is measurable, Theorem \bref{thm:Stability} 
implies that among the infinitely many functions $h_m(x):=e^{if(mx)/m}$, $x\in\R^n$, $m\in\N$, only finitely many of them can be measurable. 
\end{remark}

\section{Concluding Remarks}\label{sec:Remarks}
We conclude this paper with the following remarks. 
\begin{remark} {\bf Systems of equations:} 
Theorem \bref{thm:CauchyMeasurable} and other results regarding variations 
of the Cauchy equation \beqref{eq:Cauchy} can be generalized to systems of equations. For instance, 
given $m,n\in \N$ and a function $f:\R^n\to \R^m$ satisfying \beqref{eq:Cauchy} for all $x,y\in\R^n$, if $f=(f_1,\ldots,f_m)$ and $e^{if_k}$ is (locally) measurable for each $k\in\{1,\ldots,m\}$, then there exists an $m\times n$ matrix $C$ having real entries such that  $f(x)=Cx$ for all $x\in\R^n$. This is a simple consequence of Theorem \bref{thm:CauchyMeasurable}, because $f_k:\R^n\to\R$ satisfies \beqref{eq:Cauchy} for every $x,y\in\R^n$ and each $k$. 
\end{remark}
\begin{remark}
{\bf Infinite dimensional spaces:} It is interesting whether it is possible to generalize the results of this paper  to infinite dimensional spaces on which a measure can be defined, e.g., the ones described in \cite{Daoxing1972}, 
\cite[pp. 154--160]{Halmos1974},\cite[pp. 157--166]{Henstock1991},\cite{Yamasaki1985}. 
\end{remark}

\begin{remark}{\bf An abstract regularity condition:} \label{rem:Abstract}
An examination of the proof of Theorem \bref{thm:CauchyMeasurable} using the initial value 
approach presented in Section \bref{sec:IVP} suggests that the regularity condition can be further generalized by replacing the measure theoretic aspect in it by a more abstract one.  In what follows, we say that a triplet $(A,B,F)$ has  the periodic integral property if the following conditions hold: 
 
\begin{enumerate}[(i)]
\item $A$ is a set of real functions defined on $\R^n$;
\item $B$ is a set of functions from $\R^n$ to $\C$ and $B$ contains $\{e^{ig}: g\in A\}$;
\item $F:B\to\C$ is a functional;
\item for all $\beta\in \C$ satisfying $|\beta|=1$ and for all $h\in B$, we have $\beta h\in B$ and $F(\beta h)=\beta F(h)$;
\item the functions $x\mapsto c\cdot x$ are in $A$ for all $c\in \R^n$;
\item\label{item:VectorSpace} the set $A$ is closed under addition and under multiplication by positive rationals;
\item\label{item:F(g_y)=F(g)} there exists a basis $\{u_1,\ldots,u_n\}$ of $\R^n$ such that for all $g\in A$ satisfying the relation  $g(x+u_k)=g(x)$ for all $x\in \R^n$ and $k\in \{1,\ldots,n\}$, the function $g_y$  defined by $g_y(x):=g(x+y)$ for each $x\in \R^n$ is in $A$ for each $y\in \R^n$ 
and we have $F(e^{i g_y})=F(e^{i g})$; 
\item\label{item:alpha} for each $g\in A$ there exists some rational $\alpha>0$ such that $F(e^{i \alpha g})\neq 0$.
\end{enumerate}
As follows from Section \bref{sec:IVP}, an example of such a triplet $(A,B,F)$ is $A:=\{f:\R^n\to\R: e^{if}\,\,\textnormal{is  measurable}\}$, $B:=\{e^{ig}: g\in A\}$, $F(v):=\int_I v(x)dx$ for all $v\in A$, where $I\subset \R^n$ is an arbitrary hypercube which is spanned by the basis $\{u_1,\ldots,u_n\}$. 

Suppose now that we are given such a triplet $(A,B,F)$ satisfying the periodic integral property. We claim that in this case, if $f\in A$ and if $f$ satisfies \beqref{eq:Cauchy} for every $x,y\in\R^n$, then there exists $c\in\R^n$ such that  $f(x)=c\cdot x$ for each $x\in\R^n$. Indeed, the proof is almost word for word the proof given in Section \bref{sec:IVP}, where there are two main differences. First, in Lemma \bref{lem:UniqueRn} we need to follow only the first two paragraphs. Second,   after applying the function $t\mapsto e^{i\alpha t}$, $t\in\R$ on \beqref{eq:CauchyThm_g} (where $\alpha$ is the positive rational number from Property \beqref{item:alpha} above) we apply the functional $F$ on  both sides of the resulting equation $e^{i\alpha g(x+y)}=e^{i\alpha g(x)}e^{i\alpha g(y)}$  and use Property \beqref{item:F(g_y)=F(g)} above instead of applying the integral operator  $G(v):=\int_{I_0} v(x)dx$ on this equation (where $I_0$ is the hyper-parallelepiped spanned by the basis $\{u_1,\ldots,u_n\}$ and  defined in the proof of Lemma \bref{lem:UniqueRn}, near \beqref{eq:h_alpha}) and using Lemma \bref{lem:period2}, as done there.  

It is very interesting to find a triplet $(A,B,F)$ satisfying the periodic integral property such that  $A\neq \{f:\R^n\to\R: e^{if}\,\,\textnormal{is  measurable}\}$ and $F$ is not an integral operator (that is, $F$ does not coincide with $v\mapsto\int_{I_0} v(x)dx$ or a minor modification of this operator) or to show that this is impossible. It is of a considerable interest to find at least one triplet $(A,B,F)$  satisfying the periodic integral property in an infinite dimensional setting or to show that no such a triplet can exist. We may also try to replace $F$ by a family of functions $F_j:B\to \C$. Anyway, because there are nonlinear solutions to \beqref{eq:Cauchy}, these solutions cannot belong to $A$. Thus no matter what $A$ and $F$ are, $A$ cannot be the set of all real functions. A corresponding modification of what written in the previous lines and paragraphs holds for the topological torus.  Finally, it will be valuable to combine the above mentioned abstract regularity condition with  some existing theories of regularity conditions of functional equations, such as the 
ones surveyed and developed in  \cite{Jarai,JaraiSzekelyhidi1996jour}.
\end{remark}

\section{Appendix}\label{sec:Appendix}
\begin{proof}[{\bf Proof of Lemma \bref{lem:period2}}]
The proof is not conceptually difficult but it is somewhat technical. The main ideas behind it are illustrated in Figure \bref{fig:ParallelPeriod} below. For the sake of convenience, the proof is divided into steps. In what follows $y\in\R^n$ is fixed. \\

{\noindent \bf Step 0:} It is sufficient to assume that $I$ is right semi-open because the (right) facets of $I$ (i.e., the sets $I_j:=\{u_0+u_j+\sum_{k\in\{1,\ldots,n\}\backslash\{j\}} t_k u_k:\,\,t_k\in [0,1]\,\,\forall k\in\{1,\ldots,n\}\backslash\{j\}\}$ where $j\in\{1,\ldots,n\}$) are of measure 0 and hence adding them to or removing them from $I$ contributes nothing to the considered integrals. Moreover, it suffices to assume that $u_0=0$, because if we know that  the assertion holds for $u_0=0$ (equivalently, with the integration done over $I(0):=\{\sum_{k=1}^n t_k u_k:\,\,(t_k)_{k=1}^n\in [0,1)^n\}$) and arbitrary $y$, then a few applications of this assertion (below: first with $y+u_0$ and second with $u_0$, both instead of $y$) together with a few applications of the  change of variables formula for translations, imply the desired conclusion: 
\begin{multline*} \int_{I}h_y(x)dx=\int_{I}h(x+y)dx=\int_{I(0)+u_0}h(x+y)dx
=\int_{I(0)}h(w+u_0+y)dw\\
=\int_{I(0)}h_{u_0+y}(w)dw=\int_{I(0)}h(w)dw=\int_{I(0)}h_{u_0}(w)dw\\
=\int_{I(0)}h(w+u_0)dw=\int_{I(0)+u_0}h(x)dx=\int_{I}h(x)dx.
\end{multline*}

{\noindent \bf Step 1:} Since $U$ is a basis, we have $y=\sum_{k=1}^n y_k u_k$ where $y_k\in \R$ for each $k\in\{1,\ldots,n\}$ are uniquely determined. Given $z\in I_y:=I+y$, we have $z=x+y$ for some $x\in I$. We can write $x=\sum_{k=1}^n x_k u_k$ and $z=\sum_{k=1}^n z_k u_k$ where (because $u_0=0$ by our assumption)  $x_k\in [0,1]$ and $z_k\in\R$ are uniquely determined for all $k\in\{1,\ldots,n\}$. Suppose first that $y_k\in [0,1)$ for each $k\in\{1,\ldots,n\}$. The general case is a simple consequence of this special case (Step 9  below).\\

\begin{figure}[t]
\begin{minipage}[t]{1\textwidth}
\begin{center}
{\includegraphics[clip, scale=0.55]{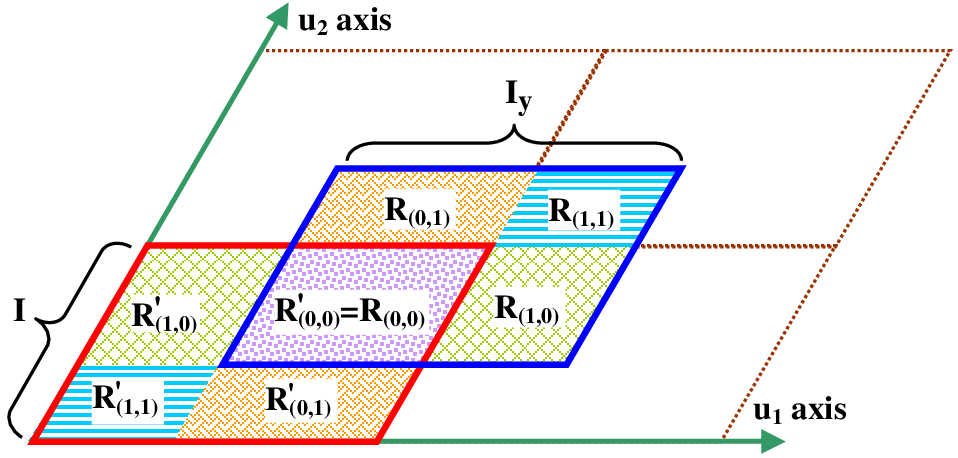}}
\end{center}
 \caption{A schematic description of the proof of Lemma \bref{lem:period2} for $n=2$. Both $I_y$ and $I$ are decomposed into unions of disjoint sub-parallelepipeds: $I_y=\cup_{s\in\{0,1\}^n}R_s$ and $I=\cup_{s\in\{0,1\}^n}R'_s$. The $U$-periodicity of $h$ and the change of variable formula for translations imply that  $\int_{R_s}h(z)dz=\int_{R'_s}h(x)dx$ for all $s\in\{0,1\}^n$. Hence (using the additivity of the integral) $\int_{I}h_y(x)dx=\int_{I_y}h(z)dz=\int_I h(x)dx$.}
\label{fig:ParallelPeriod}
\end{minipage}
\end{figure}
 
{\noindent \bf Step 2:}  Since $z_k=x_k+y_k$, it follows that $z_k\in [0,2)$ and hence either $z_k\in [0,1)$ or $z_k\in [1,2)$ for all $k\in\{1,\ldots,n\}$. Consider the greatest integer function $[\cdot]:\R\to\Z$ which assigns to each $t\in \R$ the greatest integer $[t]$ not exceeding $t$. In particular, $[t]=0$ when $t\in [0,1)$ and $[t]=1$ when $t\in [1,2)$. Define $\phi: I_y\to\{0,1\}^n$ by $\phi(z):=([z_k])_{k=1}^n$ for every $z\in I_y$, where here we used the coordinates of $z$ with respect to the basis $U$. For each $s\in \{0,1\}^n$ let $R_s:=\phi^{-1}(s)$. Since  $\phi$ is a measurable function, $R_s$ is measurable for every $s\in \{0,1\}^n$. In fact, each $R_s$ is a right semi-open hyper-parallelepiped, possibly empty (if $y_k=0$ for some $k\in\{1,\ldots,n\}$, then $R_s=\emptyset$ for all $s\in\{0,1\}^n$ whose $k$-th component is 1). For instance, if $n=2$, then 
$R_{(0,0)}=\{z_1 u_1+z_2 u_2: z_1\in [y_1,1),\,z_2\in [y_2,1)\}$, 
$R_{(0,1)}=\{z_1 u_1+z_2 u_2: z_1\in [y_1,1),\,z_2\in [1,1+y_2)\}$, 
$R_{(1,0)}=\{z_1 u_1+z_2 u_2: z_1\in [1,1+y_1),\,z_2\in [y_2,1)\}$, 
$R_{(1,1)}=\{z_1 u_1+z_2 u_2: z_1\in [1,1+y_1),\,z_2\in [1,1+y_2)\}$. See Figure \bref{fig:ParallelPeriod}.\\

{\noindent \bf Step 3:} The above construction implies that $I_y=\cup_{s\in \{0,1\}^n}R_s$ and $R_s\cap R_{p}=\emptyset$ whenever $s\neq p$, $s,p\in\{0,1\}^n$. Now we associate with each $R_s$   a certain translated copy of it which we denote by $R'_s$. The correspondence is as follows: given $s\in \{0,1\}^n$, let $R'_s:=R_s-s\cdot U$ where $s\cdot U:=\sum_{k=1}^n s_k u_k\in\R^n$. In the next steps we will see that the sets $R'_s$ form a disjoint decomposition of $I$.\\

{\noindent \bf Step 4:} We first show that $R'_s$ is contained in $I$ for each $s\in\{0,1\}^n$. Indeed, given $x\in R'_s$ we have $x=z-s\cdot U$ for some $z\in R_s$. By the definition of $R_s$ we have $[z_k]=s_k$ for all $k\in\{1,\ldots,n\}$. Fix an index $k\in\{1,\ldots,n\}$. If $s_k=0$, then according to the definition of $[\cdot]$ we have $z_k\in [0,1)$. By comparing coordinates with respect to the basis $U$ we see that $x_k=z_k-s_k=z_k\in [0,1)$. If $s_k=1$, then according to the definition of $[\cdot]$ we have $z_k\in [1,2)$. Thus  again $x_k=z_k-s_k\in [0,1)$. Because $k$ was arbitrary we conclude that indeed $x\in I$. \\

{\noindent \bf Step 5:} Here we show that $R'_s\cap R'_p=\emptyset$ whenever $s,p\in\{0,1\}^n$, $s\neq p$. Indeed, suppose for a contradiction that there exists $x\in R'_s\cap R'_p$ for two different points $p$ and $s$ in $\{0,1\}^n$. Since $s\neq p$ there exists an index $k\in\{1,\ldots,n\}$ such that $s_k\neq p_k$. Because $x\in R'_s$ we have $x=z-s\cdot U$ for some $z\in R_s$. Because $x\in R'_p$ we have $x=\tilde{z}-p\cdot U$ for some $\tilde{z}\in R_p$. In particular $z_k-s_k=x_k=\tilde{z}_k-p_k$. By the definition $R_s$ and $R_p$ we have $[z_k]=s_k$ and  $[\tilde{z}_k]=p_k$. Suppose first that $s_k=0$. Then  $x_k=z_k-0$. In addition, because $z\in R_s\subseteq I_y$ there is $v\in I$ such that $z=v+y$. Since $v\in I$ we can write $v=\sum_{j=1}^n v_j u_j$ where $v_j\in [0,1)$ for all $j\in\{1\ldots,n\}$. Hence $v_k\geq 0$ and 
\begin{equation}\label{eq:x_k>=y_k}
x_k=z_k=v_k+y_k\geq y_k. 
\end{equation}
Now, since $p_k\neq s_k$ we have $p_k=1$. This and the fact that $\tilde{z}\in R_p$  imply, according to the definition of $[\cdot]$, that $\tilde{z}_k\in [1,2)$. In addition, because $\tilde{z}\in R_p\subseteq I_y$ there is $\tilde{v}\in I$ such that $\tilde{z}=\tilde{v}+y$. Since $\tilde{v}\in I$ we can write $\tilde{v}=\sum_{j=1}^n \tilde{v}_j u_j$ where $\tilde{v}_j\in [0,1)$ for each $j\in\{1,\ldots,n\}$. Thus $\tilde{v}_k<1$ and $\tilde{z}_k<1+y_k$. Hence  $x_k=\tilde{z}_k-p_k<1+y_k-1=y_k$. This is a contradiction to \beqref{eq:x_k>=y_k}. The case $s_k=1$, namely $p_k=0$, is treated similarly and yields a similar contradiction. We conclude that  $R'_s\cap R'_p=\emptyset$, as claimed. \\

{\noindent \bf Step 6:} Here we show that $\cup_{s\in\{0,1\}^n}R'_s=I$. Indeed,  $\cup_{s\in\{0,1\}^n}R'_s\subseteq I$ by Step 4. On the other hand, let $x\in I$ be arbitrary. Then $x_k\in [0,1)$ for all $k\in \{1,\ldots,n\}$. Define $s=s(x)\in \{0,1\}^n$ as follows: given $k\in\{1,\ldots,n\}$, if $x_k\in [0,y_k)$, then  $s_k:=1$. If $x_k\in [y_k,1)$, then $s_k:=0$. We claim that $x\in R'_s$. Since $R'_s=R_s-s\cdot U$ it suffices to prove that $x+s\cdot U\in R_s$. To see this we recall that by the definition of $R_s$ and of $[\cdot]$ (Step 2) we need to show that $[x_k+s_k]=s_k$ for each $k\in\{1,\ldots,n\}$. Indeed, given $k\in\{1,\ldots,n\}$, if $x_k\in [0,y_k)$, then $s_k=1$ and hence $1\leq x_k+s_k<y_k+1<2$, so the definition of $[\cdot]$ implies that $[x_k+s_k]=1=s_k$. If $x_k\in [y_k,1)$, then $s_k=0$. Hence $x_k+s_k\in [y_k,1)\subseteq [0,1)$ and therefore $[x_k+s_k]=0=s_k$ again.\\

{\noindent \bf Step 7:} A useful observation: $h(z)=h(z-s\cdot U)$ for all $z\in \R^n$. Indeed, this is just an immediate consequence of the periodicity of $h$ with respect to the basis $U$.\\

{\noindent \bf Step 8:} Here we show that the desired integrals are equal when $y_k\in [0,1)$ for each $k\in\{1,\ldots,n\}$. First we observe that since $h$ is defined on $\R^n$ and integrable on $I$, the function $1_I h$ is integrable on $\R^n$. Hence the change of variables formula for  translations \cite[p. 171]{Jones2001book} (applied to $\re(h)$ and $\im(h)$ and hence to $h=\re(h)+i\im(h)$) implies that given $w\in \R^n$ the function  $x\mapsto 1_I(x-w)h(x-w)$, $x\in \R^n$, is integrable on $\R^n$. Let $w=\sum_{k=1}^n m_ku_k$ where $m_k$ is an integer for all $k\in\{1\ldots,n\}$. The periodicity of $h$ with respect to $U$ implies that $h(z-w)=h(z)$ for all $z\in\R^n$. Hence by the change of variables formula for translations 
\begin{multline*} 
\int_{I}h(x)dx=\int_{\R^n}1_I(x)h(x)dx=\int_{\R^n}1_I(z-w)h(z-w)dz\\
=\int_{I_w}h(z-w)dz=\int_{I_w}h(z)dz. 
\end{multline*}
Therefore $h$ is integrable (in particular, measurable) over any translated copy of $I$ where the translation vector is a linear combination of the elements in $U$ with integer coefficients. Since $\R^n$ can be represented as a countable union of translated copies of $I$ it follows that $h$ is measurable on $\R^n$ and (because of the additivity of the integral) integrable on any measurable subset of $\R^n$ contained in a finite union of translated copies of $I$. Because $I_y$ is a measurable set  contained in the union $\bigcup_{s\in \{0,1\}^n}(I+s\cdot U)$ of $2^n$ translated copies of $I$, it follows that $h$ is integrable on $I_y$. Now, by  combining this with the previous steps, with the additivity of the integral, with the periodicity of $h$ with respect to $U$, and with the change of variables formula for translations, we have 
\begin{multline*}
\int_{I}h_y(x)dx=\int_{I}h(x+y)dx=\int_{I_y}h(z)dz=\int_{\bigcup_{s\in\{0,1\}^n}R_s}h(z)dz\\
=\sum_{s\in\{0,1\}^n}\int_{R_s}h(z)dz=\sum_{s\in\{0,1\}^n}\int_{R_s}h(z-s\cdot U)dz
=\sum_{s\in\{0,1\}^n}\int_{R_s-s\cdot U}h(v)dv\\
=\sum_{s\in\{0,1\}^n}\int_{R'_s}h(v)dv=\int_{I}h(v)dv=\int_I h(x)dx.
\end{multline*}

{\noindent \bf Step 9:} Here we finish the proof by considering the case of an arbitrary $y\in\R^n$. Using the greatest integer function we can write  $y=\sum_{k=1}^n y_k u_k=\sum_{k=1}^n ([y_k]+\hat{y}_k)u_k$, where $\hat{y}_k=y_k-[y_k]\in [0,1)$ for all $k\in \{1,\ldots,n\}$. Hence, by denoting $\hat{y}:=\sum_{k=1}^n\hat{y}_k u_k$ and $[y]:=\sum_{k=1}^n[y_k] u_k$ and using the fact that $h$ is periodic with respect to the basis $U$ we have $h(x+y)=h(x+\hat{y}+[y])=h(x+\hat{y})$ for each $x\in\R^n$. Thus $h_y=h_{\hat{y}}$ and hence, from Step 8 (with $\hat{y}$ instead of $y$), we conclude that $h_y$ is integrable over $I$ and $\int_{I}h_y(x)dx=\int_{I}h_{\hat{y}}(x)dx=\int_{I}h(x)dx$. 
\end{proof}

\begin{proof}[{\bf Proof of Lemma \bref{lem:Torus}}]
Suppose first that $h$ is a characteristic function of some measurable subset $S$ of $\T^n$, namely $h=1_S$. Given $x\in\T^n$ we have $x+y\in S$ if and only if $x\in S-y:=\{s-y: s\in S\}$ for each $x\in \T^n$. Thus $h_y(x)=h(x+y)=1_S(x+y)=1_{S-y}(x)$ for all $x\in\T^n$. Therefore  $h_y$ is measurable and  $\int_{\T^n}h_y(x)d\mu(x)=\int_{\T^n}1_{S-y}(x)d\mu(x)=\mu(S-y)$. This equality and the fact that the Haar measure $\mu$ is  invariant under translations \cite[p. 285]{Cohn2013book} imply that  $\mu(S-y)=\mu(S)=\int_{\T^n}1_S(x)d\mu(x)=\int_{\T^n}h(x)d\mu(x)$. Hence $\int_{\T^n}h_y(x)d\mu(x)=\int_{\T^n}h(x)d\mu(x)$.

Suppose now that $h$ is a simple real valued function, namely, there exists $m\in\N$, mutually disjoint  measurable subsets $S_1,\ldots,S_m$ of $\T^n$, and real numbers $\beta_1,\ldots,\beta_m$ such that $h=\sum_{k=1}^m \beta_k 1_{S_k}$. From the previous paragraph $h(x+y)=\sum_{k=1}^m \beta_k 1_{S_k}(x+y)=\sum_{k=1}^m \beta_k 1_{S_k-y}(x)$ for all $x,y\in\T^n$. Hence for every $y\in\T^n$ the function $h_y$ is a finite sum of measurable functions and therefore it is measurable. Moreover, the  linearity of the integral and its translation invariance  with respect to characteristic functions established in the previous paragraph imply the desired equality: 
\begin{multline*} \int_{\T^n}h_y(x)d\mu(x)=\sum_{k=1}^m\beta_k\int_{\T^n}1_{S_k}(x+y)d\mu(x)=\sum_{k=1}^m\beta_k\int_{\T^n}1_{S_k}(x)d\mu(x)\\
=\int_{\T^n}\left(\sum_{k=1}^m\beta_k1_{S_k}(x)\right)d\mu(x)=\int_{\T^n}h(x)d\mu(x).
\end{multline*}
Now assume that $h$ is an arbitrary nonnegative measurable function. The construction of the abstract Lebesgue integral (in particular, the construction of the Haar integral) \cite[Chapter 2]{Cohn2013book}  shows that there exists an increasing sequence $(h_k)_{k=1}^{\infty}$ of nonnegative  simple functions which converges pointwise to $h$. Hence the sequence $(h_k(x+y))_{k=1}^{\infty}$ of nonnegative numbers increases to $h(x+y)$ for all $x,y\in\T^n$, i.e., for each $y\in\T^n$ the function  $h_y$ is the increasing limit of the sequence $(h_{k,y})_{k=1}^{\infty}$ of nonnegative and measurable functions (here $h_{k,y}(x):=(h_k)_y(x)=h_k(x+y)$ for all  $x,y\in\T^n$ and all $k\in\N$). Thus $h_y$ is measurable. These facts, together with the monotone  convergence theorem and the translation invariance of the integral of simple functions  established earlier, all imply that for each $y\in\T^n$ 
\begin{equation*}
\int_{\T^n}h(x)d\mu(x)=\lim_{k\to\infty}\int_{\T^n}h_k(x)d\mu(x)
=\lim_{k\to\infty}\int_{\T^n}h_{k,y}(x)d\mu(x)=\int_{\T^n}h_y(x)d\mu(x).
\end{equation*}
Finally, let $h\in L_1(\T^n)$ be arbitrary. Then $h=\re(h)+i\im(h)=(\re(h)_+-\re(h)_-)+i(\im(h)_+-\im(h)_-)$ where, as usual, $g_+:=\max\{0,g\}$ and $g_-:=-\min\{0,g\}$ for each function $g:\T^n\to\R$ and $\re(h)_+,\re(h)_-,\im(h)_+,\im(h)_-$ are all nonnegative and  integrable over $\T^n$. By what established in the previous  paragraph and the linearity of the integral we conclude that $h_y$ is measurable and $\int_{\T^n}h_y(x)d\mu(x)=\int_{\T^n}h(x)d\mu(x)$ for all $y\in\T^n$. 
\end{proof}

An examination of the proof of Lemma \bref{lem:Torus} given above shows that it can be generalized almost word for word to any locally compact topological group (instead of $\T^n$). 

\vspace{0.4cm}
\noindent{\bf Acknowledgments}\vspace{0.1cm}\\
\noindent  I would like to use this opportunity to thank several people for either helpful  remarks and/or for reading the paper and/or for sending me copies of some papers, especially to Orr Shalit, Eytan Paldi, 
Tadeusz Figiel, Roman Ger, Maciej  Sablik, Zoltan Boros, and Jolanta Olko. I also thank the referee for valuable remarks. Parts of this work were done in  various years when I have been associated with various institutes: The Technion - Israel Institute of Technology, Haifa, Israel (2010, 2016),  IMPA - The National Institute of Pure and Applied Mathematics, Rio de Janeiro, Brazil (2013), ICMC - Institute of Mathematical and Computer Sciences, University of S\~ao Paulo, S\~ao Carlos, Brazil (2014-2015), and it is an opportunity to me to thank the Gurwin Foundation, a special postdoc fellowship from IMPA  (``P\'os-doutorado de Excel\^encia''), and FAPESP.

\bibliographystyle{acm}
\bibliography{biblio}

\end{document}